\documentclass[reqno,12pt]{amsart}

\usepackage{amsmath, amsthm, amsfonts, amssymb}
\input{mathrsfs.sty}

\usepackage{amsmath}
\usepackage{amsfonts}
\usepackage{amssymb}
\usepackage{eufrak}
\usepackage{amscd}
\usepackage{amsthm}
\usepackage{amstext}
\usepackage[all]{xy}

\newcommand{\Ran}{\operatorname{Ran}}

   \theoremstyle{plain}
   \newtheorem{thm}{Theorem}
   
   \newtheorem{lem}[thm]{Lemma}
   \newtheorem{cor}[thm]{Corollary}
   \theoremstyle{definition}
   
   \newtheorem{defn}[thm]{Definition}
   
	\newtheorem*{question*}{Question}

   \theoremstyle{remark}
   
   \newtheorem{remark}[thm]{Remark}

\author{V. Manuilov}

\date{}

\address{Moscow Center for Fundamental and Applied Mathematics, Moscow State University,
Leninskie Gory 1, Moscow, 
119991, Russia}

\email{manuilov@mech.math.msu.su}

\title{Restricting operators to thick Hilbert $C^*$-submodules}

\begin{document}

\begin{abstract}
Given an essential ideal $J\subset A$ of a $C^*$-algebra $A$, and a Hilbert $C^*$-module $M$ over $A$, we place $M$ between two other Hilbert $C^*$-modules over $A$, $M_J\subset M\subset M^J$, in such a way that each submodule here is thick, i.e. its orthogonal conmplement in the greater module is trivial. We introduce the class $\mathbb B_J(M)$ of $J$-adjointable operators on a Hilbert $C^*$-module $M$ over $A$, and prove that this class isometrically embeds into the $C^*$-algebras of all adjointable operators both of $M_J$ and of $M^J$.

\end{abstract}

\maketitle

\section{Introduction}

One of the main distinctions of Hilbert $C^*$-modules over a $C^*$-algebra $A$ from Hilbert spaces is the lack of adjoints for bounded operators. Adjointable operators on a Hilbert $C^*$-module form a $C^*$-algebra, but general bounded operators are only a Banach algebra. The aim of this short paper is to show that, at least for some operators, this may be considered as an accidental misconception. To this end, given an essential ideal $J\subset A$ and a Hilbert $C^*$-module $M$ over $A$, we place $M$ between two other Hilbert $C^*$-modules over $A$, $M_J\subset M\subset M^J$, in such a way that each submodule here is thick. Recall that a submodule $N\subset M$ is thick if its orthogonal conmplement $N^\perp$ is trivial. This notion attracted attention after the results of \cite{Kaad-Skeide} showed that an extension of a functional from $N$ to $M$ need not be unique. We introduce the class $\mathbb B_J(M)$ of $J$-adjointable operators on a Hilbert $C^*$-module $M$ over $A$, and prove that this class isometrically embeds into the $C^*$-algebras of all adjointable operators both of $M_J$ and of $M^J$, showing that if a bounded operator is $J$-adjointable for some essential ideal $J\subset A$ then it becomes adjointable when we pass from $M$ to either $M_J$ or $M^J$. 

\medskip

Recall basic definitions. For a $C^*$-algebra $A$, let $M$ be a right $A$-module with a compatible structure of a linear space, equipped with a sesquilinear map $\langle\cdot,\cdot\rangle:M\times M\to A$ such that
\begin{itemize}
\item
$\langle n,ma\rangle=\langle n,m\rangle a$ for any $n,m\in M$, $a\in A$;
\item
$\langle n,m\rangle=\langle m,n\rangle^*$ for any $n,m\in M$;
\item
$\langle m,m\rangle$ is positive for any $m\in M$, and if $\langle m,m\rangle=0$ then $m=0$. 
\end{itemize} 
Then $M$ is a pre-Hilbert $C^*$-module. If it is complete with respect to the norm given by $\|m\|^2=\|\langle m,m\rangle\|$ then $M$ is a Hilbert $C^*$-module.

A bounded linear map $f:M\to A$ is a functional on $M$ if it is anti-$A$-linear, i.e. if $f(ma)=a^*f(m)$ for any $a\in A$, $m\in M$. The set of all functionals forms a Banach space which is also a right $A$-module with the action of $A$ given by $(fa)(m)=f(m)a$, $a\in A$, $m\in M$. The map $m\mapsto \widehat{m}$, where $\widehat{m}(n)=\langle n,m\rangle$, $m,n\in M$, defines an isometric inclusion $M\subset M'$ (and we identify $M$ with its image $\widehat M$ in $M'$).  
For Hilbert $C^*$-modules $M$ and $N$ over a $C^*$-algebra $A$, a bounded linear map $T:M\to N$ is called an $A$-operator (or simply an operator) if $T(ma)=T(m)a$ for any $m\in M$ and any $a\in A$. We denote the set of all operators from $M$ to $N$ by $\mathbb B(M;N)$. If $N=M$ then we use the notation $\mathbb (M)$. It is known that $\mathbb B(M)$ is a Banach algebra, but usually not a $C^*$-algebra, as operators may have no adjoint operator. An operator $T\in\mathbb B(M;N)$ is adjointable if there exists an operator $S\in\mathbb B(N;M)$ such that $\langle Sn,m\rangle_M=\langle n,Tm\rangle_N$ for any $m\in M$, $n\in N$. In this case we write $S=T^*$. The set $\mathbb B^*(M)$ of all adjointable operators in $\mathbb B(M)$ is a $C^*$-algebra, moreover, the algebra $\mathbb B(M)$ is the algebra of left multipliers for the $C^*$-algebra $\mathbb B^*(M)$.

Although there may be no adjoint operator for some $T\in\mathbb B(M;N)$, an adjoint operator patently exists as a map from $N'$ to $M'$ given by $T^*(f)(m)=f(T(m))$, $f\in N'$, $m\in M$. It even can be restricted to $N\subset N'$, but its range may go outside $M\subset M'$.    

Another feature of Hilbert $C^*$-modules is that, unlike Hilbert spaces, they may contain the so-called thick submodules. A closed submodule $N\subset M$ is thick if its orthogonal complement $N^\perp=\{m\in M:\langle m,n\rangle=0 \ \forall n\in N\}$ in $M$ is zero. 

More details on Hilbert $C^*$-modules can be found in \cite{Lance} or \cite{MT}.

\section{Making a sandwich of Hilbert $C^*$-modules from an essential ideal}

By an ideal in a $C^*$-algebra we always mean a closed two-sided ideal.

Let $J\subset A$ be an essential ideal, i.e. $J\cap I\neq 0$ for any non-zero ideal $I$, or, equivalently, $aJ=0$ implies $a=0$ for $a\in A$. 

For a Hilbert $C^*$-module $M$ over $A$, the closure $\overline{MJ}$ of the set 
$$
MJ=\Bigl\{\sum\nolimits_i{m_ix_i}:m_i\in M,x_i\in J\Bigr\}
$$ 
(the sums are finite here) is a $C^*$-submodule in $M$. We shall denote it by $M_J$.

\begin{lem}
If $J$ is an essential ideal then the $C^*$-submodule $M_J$ is thick.

\end{lem}
\begin{proof}
Let $n\perp M_J$. Then $0=\langle n,mx\rangle=\langle n,m\rangle x$ for any $m\in M$ and any $x\in J$. As $J$ is essential, $\langle n,m\rangle=0$ for any $m\in M$, whence $n=0$.
\end{proof}

Set 
$$
M_1(J)=\{f\in M':fx\in M\ \forall x\in J\}. 
$$
Clearly, $M\subset M_1(J)\subset M'$, and $M_1(J)$ is a Banach module over $A$. Moreover, it is a Hilbert $C^*$-module over the multiplier $C^*$-algebra $M(J)$ of $J$: indeed, for any $f,g\in M_1(J)$ we may set 
$$
\lambda_{f,g}(x)=g(f(x))^*, \quad \rho_{f,g}(x)=f(g(x^*)), \quad x\in J, 
$$
then $\lambda_{f,g}$ (resp., $\rho_{f,g}$) is a left (resp., right) multiplier of $J$, and $m_{f,g}=(\lambda_{f,g},\rho_{f,g})\in M(J)$ is a double multiplier, so one may set $\langle g,f\rangle=m_{f,g}$, which turns $M_1(J)$ into a Hilbert $C^*$-module over $M(J)$. This construction follows the construction of an extension of a Hilbert $C^*$-module by D. Bakic and B. Guljas \cite{B-G}. 

When $J$ is essential, it induces the inclusion $\alpha:A\to M(J)$. We may identify $A$ with the range of $\alpha$.

Set 
$$
M^J=\{f\in M_1(J):\langle g,f\rangle\in\Ran\alpha\ \forall g\in M_1(J)\}. 
$$
Then $M^J$ is a Hilbert $C^*$-module over $A$.

Thus, we have 
$$
M_J\subset M\subset M^J.
$$   

\begin{lem}\label{Lem2}
$(M^J)_J=M_J$.

\end{lem}
\begin{proof}
Any element $x$ of a $C^*$-algebra $J$ can be written as a product $x=x_1x_2$ with $x_1,x_2\in J$. If $f\in M^J$ then $fx\in M$ for any $x\in J$, so $f(x_1)\in M$ and $f(x)=f(x_1)x_2\in M_J$.

\end{proof}

\begin{cor}
$M_J$ and $M$ are thick submodules in $M^J$. 

\end{cor}

\begin{remark}
We do not know if $(M_J)^J=M^J$.

\end{remark}

\section{$J$-adjointable operators}

\begin{defn}
Let $M,N$ be Hilbert $C^*$-modules over $A$, and let $J\subset A$ be an essential ideal.
A bounded operator  $T\in\mathbb B(M;N)$ is $J$-adjointable if $T^*(\widehat{n})x\in \widehat M$ for any $n\in N$ and any $x\in J$. 

\end{defn}

Note that if $T$ is adjointable then, for any $n\in N$, $T^*(\widehat{n})=\widehat{m}$ for some $m\in M$, hence $T$ is $J$-adjointable.

\begin{lem}
Let $T\in\mathbb B(M,N)$. Then $T(M_J)\subset M_J$.

\end{lem}
\begin{proof}
As $T(mx)=T(mx_1x_2)=T(mx_1)x_2$, $x,x_1,x_2\in J$, we see that $T(MJ)\subset NJ$. Continuiuty of $T$ finishes the proof.

\end{proof}

Let $\underline{T}=T|_{M_J}$ be the restriction of $T$ onto $M_J$.

\begin{lem}
If $T$ is $J$-adjointable then $\underline{T}$ is adjointable.

\end{lem}
\begin{proof}
The formula $\widehat{S(nx)}=T^*(\widehat{n}x)$, $n\in N$, $x\in J$, defines an $A$-linear operator $S$ from $N_J$ to $M_J$. 
As 
\begin{eqnarray*}
\langle my, S(nx)\rangle&=&\widehat{S(nx)}(my)\\
&=&T^*(\widehat{nx})(my)\\
&=&\widehat{nx}(T(my))\\
&=&\langle nx,T(my)\rangle, 
\end{eqnarray*}
where $m\in M$, $y\in J$, we see that $S$ is the adjoint for $\underline{T}$.
\end{proof}

Now let us define an extension $\overline{T}$ of a $J$-adjointable operator $T$ to $M^J$.
Let $f\in M^J$, let $T$ be a $J$-adjointable operator from $M$ to $N$, and let $n\in N$. Then put 
$$
\overline{T}(f)(n)=m_{f,T^*(\widehat{n})}.
$$ 
As $T$ is $J$-adjointable, $T^*(\widehat{n})\in M_1(J)$, hence $m_{f,T^*(\widehat{n})}\in\alpha(A)\subset M(J)$. $A$-linearity and boundedness of $\overline{T}(f):N\to A$ is obvious, so $\overline{T}(f)$ is a functional on $N$, i.e. $\overline{T}(f)\in N'$. 

Set also $S(g)(m)=m_{T(m),g}$, where $g\in N_1(J)$, $m\in M$. The same argument shows that $S(g)\in M'$. As any $x\in J$ can be written as a product $x=x_1x_2$, $x_1,x_2\in J$, $\overline{T}(f)\in N_1(J)$ and $S(g)\in M_1(J)$. 

\begin{lem}
$\overline{T}(f)\in N^J$, $S(g)\in M^J$, and $S$ is the adjoint for $\overline{T}$. 

\end{lem}
\begin{proof}
Let $g\in N_1(J)$, $x,y\in J$. Then 
\begin{eqnarray*}
x^*m_{\overline{T}(f),g}y&=&x^*\langle g,\overline{T}(f)\rangle y\\
&=&\langle gx,\overline{T}(f)y\rangle\\
&=&\langle gx, T(fy)\rangle\\
&=&\langle T^*(gx),fy\rangle\\
&=&x^*\langle T^*(g),f\rangle y\\
&=&x^* m_{f,T^*(g)}y. 
\end{eqnarray*}
As $J$ is essential, this implies 
\begin{equation}\label{e2}
m_{\overline{T}(f),g}=m_{f,T^*(g)}. 
\end{equation}
We have $T^*(g)\in M_1(J)$, $f\in M^J$, hence $m_{f,T^*(g)}\in \Ran\alpha$, hence $m_{\overline{T}(f),g}\in\Ran\alpha$ for any $g\in N_1(J)$, hence $\overline{T}(f)\in N^J$. Similarly, $S(g)\in M^J$ for any $g\in N^J$. It  follows from (\ref{e2}) that $S$ is the adjoint for $\overline{T}$. 
\end{proof}

In the case when $N=M$ the set of $J$-adjointable operators is a Banach algebra. 
\begin{lem}
Let $T,S\in \mathbb B(M)$ be $J$-adjointable. Then
\begin{enumerate}
\item 
$TS$ is $J$-adjointable;
\item
$\underline{TS}=\underline{T}\cdot\underline{S}$ and $\overline{TS}=\overline{T}\cdot\overline{S}$.
\end{enumerate}

\end{lem}
\begin{proof}
Let $x=x_1x_2$, $x,x_1,x_2\in J$, $n\in M$. As $T$ is $J$-adjointable, $T^*(\widehat{n}x_1)=\widehat m$ for some $m\in M$. Then 
$$
(TS)^*(\widehat{n}x)=S^*(T^*(\widehat{n}x_1)x_2)=S^*(\widehat{m}x_2)\in\widehat{M},
$$  
hence $TS$ is $J$-adjointable. As both maps $T\mapsto\underline{T}$ and $T\mapsto\overline{T}$ are restrictions (the latter is the restriction of $T^*$), these maps are homomorphisms.
\end{proof}

Denote by $\mathbb B_J(M)$ the Banach algebra of $J$-adjointable operators on $M$. Summing up, we have constructed homomorphisms 
$$
\underline{\alpha}:\mathbb B_J(M)\to \mathbb B^*(M_J) \quad\mbox{and}\quad \overline{\alpha}:\mathbb B_J(M)\to \mathbb B^*(M^J). 
$$
Although $\mathbb B_J(M)$ is not involutive, if $T\in \mathbb B_J(M)$ is adjointable then $\underline{\alpha}(T^*)=\underline{T}^*$ and $\overline{\alpha}(T^*)=\overline{T}^*$.
 
By Lemma \ref{Lem2}, we also have the restriction map $\alpha:\mathbb B^*(M^J)\to\mathbb B^*(M_J)$, which is a $*$-homomorphism.

\begin{thm}
The maps $\alpha$ and $\underline{\alpha}$ are isometric injective homomorphisms.

\end{thm}
\begin{proof}
As the proofs for $\alpha$ and $\underline{\alpha}$ are the same, we give only the latter.

Injectivity of $\underline{\alpha}$ would follow from isometricity, but can be proved directly. If $\underline{T}=0$ then $T(mx)=0$ for any $m\in M$, $x\in J$, therefore, $\langle T(mx),T(mx)\rangle=x^*\langle Tm,Tm\rangle x=0$, and essentiality of $J$ implies that $\langle Tm,Tm\rangle=0$ for any $m\in M$, hence $T=0$.

Let us show that $\|T\|=\|\underline{T}\|$ for any $T\in\mathbb B_J(M)$. As $\underline{T}$ is obtained by restricting $T$, we have $\|\underline{T}\|\leq\|T\|$. Recall that, by Remark 2.9 of \cite{PaschkeTAMS}, 
$$
\|T\|^2=\inf\{K:\langle Tm,Tm\rangle\leq K\langle m,m\rangle\ \ \forall m\in M\}, 
$$
so the estimate $\|T\|\leq\|\underline{T}\|$ would follow if we prove that if $x^*\langle Tm,Tm\rangle x\leq Kx^*\langle m,m\rangle x$ holds for any $x\in J$ then $\langle Tm,Tm\rangle\leq K\langle m,m\rangle$. Write $a=\langle Tm,Tm\rangle$, $b=K\langle m,m\rangle$, $c=a-b\in A$. It would suffice to prove that if $x^*cx\geq 0$ for any $x\in J$ then $c\geq 0$. Write $c=c_+-c_-$, where $c_+,c_-\in A$ are positive and $c_+c_-=0$. Set $x=c_-^{1/2}y$, $y\in J$. If $x^*cx\geq 0$ then $y^*c_-^{1/2}cc_-^{1/2}y\geq 0$ for any $y\in J$. But $y^*c_-^{1/2}cc_-^{1/2}y=-y^*c_-^2y$, thus we have $y^*c_-^2y=(c_-y)^*(c_-y)\leq 0$, hence $c_-y=0$ for any $y\in J$. As $J$ is essential, we conclude that $c_-=0$, and $c=c_+$ is positive.    
\end{proof}

\section{Examples} 

1. Let $A=\mathbb B(H)$ be the $C^*$-algebra of all bounded operators on a separable Hilbert space $H$, and let $J=\mathbb K(H)$ be its ideal of compact operators. Let $M,N$ be Hilbert $C^*$-modules over $A$, $T\in\mathbb B(M,N)$. It was shown in \cite{M-MN} that for any functional $f\in M'$ we have $fx\in\widehat{M}$ for any $x\in J$. As $T^*(n)\in M'$ for any $n\in N$, we conclude that any bounded operator $T$ is $J$-adjointable.

2. Let $A=C[0,1]$, and let $M=N=l_2(A)$ be the standard Hilbert $C^*$-module over $A$. It was shown in \cite{Kaad-Skeide} (see also \cite{M-MN}) that there exists a functional $f\in M'$ such that if $fa\in M$ for some $a\in A$ then $a=0$. Let $T=T_{k,f}:M\to M$ be given by $T(m)=kf(m)^*$ for some $k\in M$. Then 
\begin{eqnarray*}
T^*(\widehat{n})(m)&=&\widehat{n}(Tm)=\widehat{n}(kf(m)^*)\\
&=&f(m)\widehat{n}(k)=f(m)\langle k,n\rangle\\
&=&(f\langle k,n\rangle)(m),
\end{eqnarray*}
hence $T^*(\widehat{n})=f\langle k,n\rangle$. If $T^*(\widehat{n})\in\widehat{M}$ for any $n\in M$ then $\langle k,n\rangle=0$, which is patently false for $n=k$. Thus, $T_{k,f}$ is not $J$-adjointable for any essential ideal $J$. On the other hand, there exists a functional $f\in M'$ and an essential ideal $J\subset A$ such that $fx\in M$ for any $x\in J$. For such $f$ the operator $T_{k,f}$ is $J$-adjointable.

\end{document}